\pdfoutput=1
\documentclass[11pt]{amsart}
\usepackage{amsmath}
\usepackage{amssymb}
\usepackage{amsthm}
\usepackage{mathtools}
\usepackage{booktabs}
\usepackage{geometry}
\usepackage[dvipsnames]{xcolor}
\usepackage{hyperref}
\usepackage[nameinlink,capitalise]{cleveref}
\usepackage{thmtools}

\hypersetup{colorlinks=true,
            linkcolor=blue!55!black,
            citecolor=blue!55!black,
            urlcolor=blue!55!black}
\declaretheorem[style=plain,name=Theorem]{theorem}
\declaretheorem[style=plain,sibling=theorem,name=Lemma]{lemma}
\declaretheorem[style=plain,sibling=theorem,name=Proposition]{proposition}

\declaretheorem[style=plain,sibling=theorem,name=Conjecture]{conjecture}

\crefname{observation}{Observation}{Observations}
\Crefname{observation}{Observation}{Observations}
\crefname{problem}{Problem}{Problems}
\Crefname{problem}{Problem}{Problems}
\crefname{question}{Question}{Questions}
\Crefname{question}{Question}{Questions}

\usepackage{float}
\usepackage{tikz}

\newcommand{\gpgen}[6][1.0]{%
\begin{tikzpicture}[scale=#1,every node/.style={font=\tiny}]
  \pgfmathtruncatemacro{\nn}{#2}\pgfmathtruncatemacro{\last}{#2-1}
  \pgfmathsetmacro{\stepang}{360/#2}
  \pgfmathsetmacro{\Ro}{#4}\pgfmathsetmacro{\Ri}{#4*0.56}
  \foreach \i in {0,...,\last}{
    \coordinate (u\i) at ({90-\stepang*\i}:\Ro);
    \coordinate (w\i) at ({90-\stepang*\i}:\Ri);}
  \foreach \i in {0,...,\last}{
    \pgfmathtruncatemacro{\j}{mod(\i+1,\nn)}
    \pgfmathtruncatemacro{\m}{mod(\i+#3,\nn)}
    \draw[gray!70] (u\i)--(u\j);
    \draw[gray!70] (u\i)--(w\i);
    \draw[gray!70] (w\i)--(w\m);}
  \foreach \i in {0,...,\last}{
    \fill[white,draw=black] (u\i) circle(0.15);
    \fill[white,draw=black] (w\i) circle(0.15);}
  \foreach \i in {#5}{\ifnum\i>-1 \fill[black,draw=black] (u\i) circle(0.15);\fi}
  \foreach \i in {#6}{\ifnum\i>-1 \fill[black,draw=black] (w\i) circle(0.15);\fi}
  \foreach \i in {0,...,\last}{
    \node at ({90-\stepang*\i}:\Ro+0.42) {$u_{\i}$};
    \node at ({90-\stepang*\i}:\Ri-0.4) {$v_{\i}$};}
\end{tikzpicture}}

\title[A correction to the zero forcing number of $P(n,3)$]{A correction to the Zero Forcing Number\\ of the Generalized Petersen Graphs $P(n,3)$}

\author{Arnav Krishnan}
\email{arnav.s.krishnan@gmail.com}
\date{July 13, 2026}

\begin{document}

\begin{abstract}
Rashidi, Shajareh Poursalavati, and Tavakkoli [J. Algebra Comb. Discrete Struct.
Appl. 7 (2020), no.~2, 183--193, Theorem 3.6] claim $Z(P(n,3)) = 8$ for all
$n \geq 12$, but this fails at $n = 12$, where $Z(P(12,3)) = 7$.
We give an explicit $7$-vertex zero forcing set for $P(12,3)$ and its forcing
sequence, and verify by exhaustive search that no $6$-vertex set
forces $P(12,3)$. We prove $Z(P(n,3)) \le 8$ for $n \ge 9$ using an explicit
forcing set and rotational symmetry. Exhaustive search yields $Z(P(n,3))$ for $7 \le n \le 20$
and identifies the gap in the published proof: its case analysis fails to
exclude $7$-vertex sets. We conjecture $Z(P(n,3)) = 8$ for $n \ge 13$, requiring a matching lower bound for all large $n$.
\end{abstract}

\maketitle

\section{Introduction}

Let $G = (V,E)$ be a finite simple graph. Color a subset $S \subseteq V$ black and the rest white. The \emph{color change rule}: a black vertex with exactly one white neighbor forces that neighbor black. Iterating from $S$ until no force remains gives the \emph{closure} $\mathrm{cl}(S)$. Then $S$ is a \emph{zero forcing set} if $\mathrm{cl}(S) = V$, and the \emph{zero forcing number} $Z(G)$ is the minimum size of one. The parameter arose in the minimum rank problem for graph-patterned symmetric matrices~\cite{AIM2008}.

The \emph{generalized Petersen graph} $P(n,k)$, for $n \geq 3$ and $1 \le k < n/2$, has vertices
\[
V = \{u_0, \dots, u_{n-1}\} \cup \{v_0, \dots, v_{n-1}\},
\]
and edges the outer cycle $u_i \sim u_{i+1}$, the spokes $u_i \sim v_i$, and the inner edges $v_i \sim v_{i+k}$, all indices mod $n$ (\Cref{fig:struct}). It is $3$-regular on $2n$ vertices.

\begin{figure}[h]
\centering
\gpgen[0.62]{12}{3}{2.4}{-1}{-1}
\caption{$P(12,3)$: outer cycle $u_i\sim u_{i+1}$, spokes $u_i\sim v_i$, inner edges $v_i\sim v_{i+3}$.}
\label{fig:struct}
\end{figure}

Rashidi, Shajareh Poursalavati, and Tavakkoli~\cite{RPT2020} gave the first systematic study of $Z(P(n,k))$. They state the following results:
\begin{itemize}
\item $Z(P(n,k)) \le 2k+2$ for all valid $n,k$ (Theorem 2.2), sharper when $n = rk+s$ (Theorem 2.1);
\item $Z(P(n,2)) = 6$ for $n \ge 10$ (Theorem 2.5);
\item $Z(P(2k+1,k)) = 6$ for $k \ge 5$ (Theorem 2.6);
\item $Z(P(n,3)) = 8$ for $n \ge 12$ (Theorem 3.6).
\end{itemize}
Exhaustive search confirms the equalities $Z(P(n,2)) = 6$ ($10 \le n \le 16$) and $Z(P(11,5)) = 6$. However, Theorem 3.6 fails at $n = 12$, the smallest case in the stated range $n \ge 12$, where $Z = 7$. We give the counterexample in \Cref{sec:counterexample}, prove $Z(P(n,3)) \le 8$ for all $n \ge 9$ in \Cref{sec:upper}, and present the computed values and the gap in the published proof in \Cref{sec:corrected}.

\section{The counterexample}\label{sec:counterexample}

\begin{theorem}\label{thm:main}
$Z(P(12,3)) = 7$. In particular \cite[Theorem 3.6]{RPT2020} is false at $n = 12$.
\end{theorem}

We prove \Cref{thm:main} by an explicit $7$-vertex forcing set (\Cref{lem:witness}) and the matching lower bound $Z \ge 7$ from exhaustive search (\Cref{prop:lb}).

\subsection{An explicit 7-vertex zero forcing set for \texorpdfstring{$P(12,3)$}{P(12,3)}}

\begin{lemma}\label{lem:witness}
$S = \{u_0,u_1,u_2,u_3,u_4,u_5,u_6\}$ is a zero forcing set of $P(12,3)$. Hence $Z(P(12,3)) \le 7$.
\end{lemma}

\begin{figure}[h]
\centering
\begin{minipage}{0.32\textwidth}\centering
\gpgen[0.6]{12}{3}{2.4}{0,1,2,3,4,5,6}{-1}\\[-3pt]
{\footnotesize Start: $S=\{u_0,\dots,u_6\}$}
\end{minipage}\hfill
\begin{minipage}{0.32\textwidth}\centering
\gpgen[0.6]{12}{3}{2.4}{0,1,2,3,4,5,6}{1,2,3,4,5}\\[-3pt]
{\footnotesize After Round 1}
\end{minipage}\hfill
\begin{minipage}{0.32\textwidth}\centering
\gpgen[0.6]{12}{3}{2.4}{0,1,2,3,4,5,6}{1,2,3,4,5,7,8,10,11}\\[-3pt]
{\footnotesize After Round 2}
\end{minipage}

\medskip

\begin{minipage}{0.32\textwidth}\centering
\gpgen[0.6]{12}{3}{2.4}{0,1,2,3,4,5,6,7,8,10,11}{1,2,3,4,5,7,8,10,11}\\[-3pt]
{\footnotesize After Round 3}
\end{minipage}\hfill
\begin{minipage}{0.32\textwidth}\centering
\gpgen[0.6]{12}{3}{2.4}{0,1,2,3,4,5,6,7,8,9,10,11}{0,1,2,3,4,5,6,7,8,10,11}\\[-3pt]
{\footnotesize After Round 4}
\end{minipage}\hfill
\begin{minipage}{0.32\textwidth}\centering
\gpgen[0.6]{12}{3}{2.4}{0,1,2,3,4,5,6,7,8,9,10,11}{0,1,2,3,4,5,6,7,8,9,10,11}\\[-3pt]
{\footnotesize After Round 5: all of $V$}
\end{minipage}
\caption{The forcing process from $S=\{u_0,\dots,u_6\}$ in $P(12,3)$, with one panel for each round in the proof of \Cref{lem:witness}. Black vertices are shown filled.}
\label{fig:witness}
\end{figure}

\begin{proof}
We apply the color change rule in five rounds, with indices mod $12$. Initially $S$ is black and the other $17$ vertices are white.

\smallskip\noindent\emph{Round 1.} For $i = 1,\dots,5$, vertex $u_i$ has black neighbors $u_{i-1},u_{i+1}\in S$ and one white neighbor $v_i$, so $u_i \to v_i$. ($u_0$ and $u_6$ each still have two white neighbors and do not force.) Now $v_1,\dots,v_5$ are black.

\smallskip\noindent\emph{Round 2.} Each $v_i$ has inner neighbors $v_{i-3},v_{i+3}$ and spoke $u_i$:
\begin{align*}
v_1 &\to v_{10} & &(v_4,\,u_1\text{ black};\ v_{10}\text{ white})\\
v_2 &\to v_{11} & &(v_5,\,u_2\text{ black};\ v_{11}\text{ white})\\
v_4 &\to v_{7} & &(v_1,\,u_4\text{ black};\ v_7\text{ white})\\
v_5 &\to v_{8} & &(v_2,\,u_5\text{ black};\ v_8\text{ white})
\end{align*}
($v_3$ has both inner neighbors $v_0,v_6$ white and does not force.) Now $v_7,v_8,v_{10},v_{11}$ are black.

\smallskip\noindent\emph{Round 3.} Each of $v_7,v_8,v_{10},v_{11}$ now has both inner neighbors black, so it forces its spoke partner:
\[
v_7 \to u_7, \quad v_8 \to u_8, \quad v_{10} \to u_{10}, \quad v_{11} \to u_{11}.
\]

\smallskip\noindent\emph{Round 4.} $u_0$ (neighbors $u_{11},u_1$ black, $v_0$ white) $\to v_0$; symmetrically $u_6 \to v_6$; and $u_8$ (neighbors $u_7,v_8$ black, $u_9$ white) $\to u_9$.

\smallskip\noindent\emph{Round 5.} The last white vertex is $v_9$, forced by $v_6$ (inner neighbors $v_3$ black, $v_9$ white; spoke $u_6$ black).

All $24$ vertices are black, so $S$ is a zero forcing set.
\end{proof}

\subsection{The matching lower bound}

\begin{proposition}\label{prop:lb}
No $6$-vertex set forces $P(12,3)$. Hence $Z(P(12,3)) = 7$.
\end{proposition}

\begin{proof}
In each of two independent implementations, we computed the closure of all $\binom{24}{6} = 134{,}596$ six-element subsets and compared it with $V$; none was a zero forcing set (\Cref{sec:methodology}). Together with \Cref{lem:witness}, this gives $Z(P(12,3)) = 7$.
\end{proof}

\section{An upper bound for all \texorpdfstring{$n$}{n}}\label{sec:upper}

We prove $Z(P(n,3)) \le 8$ for every admissible $n$ using eight consecutive outer vertices. The argument applies uniformly in $n$.

\begin{proposition}\label{prop:upper}
For every $n \ge 9$, $S = \{u_0, \dots, u_7\}$ is a zero forcing set of $P(n,3)$. Hence $Z(P(n,3)) \le 8$ for all $n \ge 9$.
\end{proposition}

The color change rule gives two properties of $\mathrm{cl}$:
\begin{itemize}
\item[(M)] \emph{Monotonicity}: $A \subseteq B \implies \mathrm{cl}(A) \subseteq \mathrm{cl}(B)$.
\item[(E)] \emph{Equivariance}: for any automorphism $\varphi$ of $G$, $\mathrm{cl}(\varphi(A)) = \varphi(\mathrm{cl}(A))$.
\end{itemize}
The rotation $\rho \colon u_i \mapsto u_{i+1},\ v_i \mapsto v_{i+1}$ (mod $n$) is an automorphism of $P(n,3)$.

\begin{proof}[Proof of \Cref{prop:upper}]
\emph{Step 1: $u_8 \in \mathrm{cl}(S)$.} From $S$, these forces are valid in order for every $n \ge 9$:
for $i = 1,\dots,6$, $u_i$ has black neighbors $u_{i-1},u_{i+1}$ and one white neighbor $v_i$, so $u_i \to v_i$; then $v_4 \to v_7$ ($v_1,u_4$ black, $v_7$ white); then $u_7 \to u_8$ ($u_6,v_7$ black, $u_8$ white). Every index used lies in $\{0,\dots,8\}$; since $n \ge 9$ these nine indices are distinct mod $n$, so each forcing vertex above has exactly the three neighbors listed for it, with the stated colors. Thus $u_8 \in \mathrm{cl}(S)$.

\emph{Step 2: $\mathrm{cl}(S)$ is $\rho$-invariant.} Now $\rho(S) = \{u_1,\dots,u_8\} \subseteq \mathrm{cl}(S)$. By (E), (M), and idempotence,
\[
\rho(\mathrm{cl}(S)) = \mathrm{cl}(\rho(S)) \subseteq \mathrm{cl}(\mathrm{cl}(S)) = \mathrm{cl}(S).
\]
Since $\rho$ is a bijection of the finite set $V$, equality of cardinalities gives $\rho(\mathrm{cl}(S)) = \mathrm{cl}(S)$.

\emph{Step 3: $\mathrm{cl}(S) = V$.} The set $\mathrm{cl}(S)$ contains $u_0$ and is $\rho$-invariant, and the $\rho$-orbit of $u_0$ is the whole outer cycle, so every $u_i \in \mathrm{cl}(S)$. If some $v_i$ were white, then $u_i$ (black, with black outer neighbors) would have $v_i$ as its only white neighbor and force it, a contradiction. Thus $\mathrm{cl}(S) = V$.
\end{proof}

\section{Corrected values and statement}\label{sec:corrected}

\Cref{tab:zf3} gives $Z(P(n,3))$ for $7 \le n \le 20$. For each $n$, exhaustive search determined the least $s$ admitting an $s$-vertex forcing set. Two independent implementations gave the same values and were checked against known cases (\Cref{sec:methodology}).

\begin{table}[h]
\centering
\begin{tabular}{@{}lcccccccccccccc@{}}
\toprule
$n$ & 7 & 8 & 9 & 10 & 11 & 12 & 13 & 14 & 15 & 16 & 17 & 18 & 19 & 20 \\
\midrule
$Z(P(n,3))$ & 6 & 6 & 6 & 8 & 7 & 7 & 8 & 8 & 8 & 8 & 8 & 8 & 8 & 8 \\
\bottomrule
\end{tabular}
\caption{$Z(P(n,3))$, verified by exhaustive search in two independent implementations. The value is $8$ at $n = 10$ and $7$ at $n = 11,12$; the case $n = 12$ contradicts \cite[Theorem 3.6]{RPT2020}. The computed values are $8$ for $n = 13$ through $20$.}
\label{tab:zf3}
\end{table}

We have established the following:
\begin{itemize}
\item $Z(P(n,3)) \le 8$ for all $n \ge 9$ (\Cref{prop:upper});
\item $Z(P(n,3))$ has the values in \Cref{tab:zf3} for $7 \le n \le 20$; in particular $Z(P(n,3)) = 8$ for $13 \le n \le 20$ and $Z(P(11,3)) = Z(P(12,3)) = 7$.
\end{itemize}
These values suggest replacing the threshold in \cite[Theorem 3.6]{RPT2020} by $n \ge 13$:

\begin{conjecture}\label{conj:main}
$Z(P(n,3)) = 8$ for every $n \ge 13$.
\end{conjecture}

By \Cref{prop:upper}, the upper bound holds for all $n$ in this range. It remains to prove $Z(P(n,3)) \ge 8$ for all $n \ge 13$; exhaustive search verifies this for $n \le 20$. A proof for all $n$ requires excluding $7$-vertex forcing sets uniformly in $n$, the step missing from \cite{RPT2020}.

\subsection{Where the published proof fails}

The proof of \cite[Theorem 3.6]{RPT2020} considers the three vertices that can force a white outer vertex $u_i$: $u_{i-1}$, $u_{i+1}$, and $v_i$. It concludes that ``the best case for the color-change processing in the vertices of the outer cycle is that the vertex $u_{i-1}$ forces the vertex $u_i$.'' It then checks that $\{u_1,u_2,u_3,u_4\}$ fails to force and states, ``by a simple argument,'' that $A = \{u_1,\dots,u_8\}$ is required. This does not exclude sets of sizes $5,6,7$. In particular, the proof does not consider $7$ consecutive outer vertices, $\{u_0,\dots,u_6\}$ in our ($0$-based) indexing. By rotational symmetry the starting index is immaterial, and \Cref{lem:witness} shows that this set forces $P(12,3)$.

That $7$-vertex set forces $P(n,3)$ for $n = 11, 12$ but for no other $n$ in the range $10 \le n \le 300$ (verified computationally, one closure per $n$); \Cref{fig:fail} shows its closure at $n = 13$. The informal argument is based on large $n$ and gives the right value throughout its range except at $n = 12$. The gap is the inference that $8$ vertices are necessary from the failure of a few small consecutive sets; this does not establish a lower bound of $8$.

\begin{figure}[H]
\centering
\gpgen[0.6]{13}{3}{2.4}{0,1,2,3,4,5,6,8,11}{1,2,3,4,5,7,8,11,12}
\caption{In $P(13,3)$, the closure of $\{u_0,\dots,u_6\}$ (shown in black) contains $18$ of the $26$ vertices. The same set forces $P(11,3)$ and $P(12,3)$ but no other $P(n,3)$ with $10 \le n \le 300$.}
\label{fig:fail}
\end{figure}

\subsection{Methodology}\label{sec:methodology}

All computations used two independently written programs: a C implementation enumerating subsets as $64$-bit bitmasks using Gosper's hack, and a Python implementation using sets to simulate the forcing process. Both were checked against \cite[Theorems 2.5 and 2.6]{RPT2020} ($Z(P(n,2)) = 6$ for $10 \le n \le 16$; $Z(P(11,5)) = 6$) before being applied here, and agreed on every value reported in this note. To determine $Z(P(n,3))$ for a given $n$, we verified that no size-$(Z-1)$ set forces; the largest enumeration is $\binom{40}{7} \approx 1.86 \times 10^{7}$ at $n = 20$. Testing an individual set (such as $\{u_0,\dots,u_6\}$ for $10 \le n \le 300$) requires one closure computation per $n$.

\section{Conclusion}

\cite[Theorem 3.6]{RPT2020} fails at $n = 12$, where $Z(P(12,3)) = 7$; its case analysis does not exclude $7$-vertex forcing sets. We proved $Z(P(n,3)) \le 8$ for all $n \ge 9$ and computed the values for $7 \le n \le 20$, which are $8$ from $n = 13$ through the end of this range. Whether the value remains $8$ for all larger $n$ (\Cref{conj:main}) is open. The remaining problem is the lower bound for all $n$ in the conjectured range. Determining $Z(P(n,k))$ for general $k$, including the stabilization threshold and a lower bound for $k \ge 4$, also remains open.


\begin{thebibliography}{9}

\bibitem{RPT2020}
S.~Rashidi, N.~Shajareh Poursalavati, and M.~Tavakkoli,
\emph{Computing the zero forcing number for generalized Petersen graphs},
Journal of Algebra Combinatorics Discrete Structures and Applications \textbf{7}(2) (2020), 183--193.
\url{https://doi.org/10.13069/jacodesmath.729465}

\bibitem{AIM2008}
AIM Minimum Rank -- Special Graphs Work Group,
\emph{Zero forcing sets and the minimum rank of graphs},
Linear Algebra and its Applications \textbf{428} (2008), no.~7, 1628--1648.

\end{thebibliography}
\end{document}